\documentclass{amsart}
\usepackage{amsmath,amssymb,amsthm,amscd, xypic}
\usepackage{color}
\usepackage{enumerate}
\usepackage[T1]{fontenc} % Makes the rendering of "ť" correct.
\usepackage{hyperref}

\title{Migration of silting objects via adjoint pairs}

\author{Simion Breaz, Andrei M\u arcu\c s and George Ciprian Modoi}
\address[Simion Breaz, Andrei M\u arcu\c s and George Ciprian Modoi]{Babe\c s--Bolyai University, Faculty of Mathematics and Computer Science \\
1, Mihail Kog\u alniceanu, 400084 Cluj--Napoca, Romania}
\email[Simion Breaz]{bodo@math.ubbcluj.ro}
\email[Andrei M\u arcu\c s]{amarcus@math.ubbcluj.ro}
\email[George Ciprian Modoi]{cmodoi@math.ubbcluj.ro}
\thanks{This work was supported by a grant of the Ministry of Research, Innovation and Digitization, CNCS-UEFISCDI, project number PN-IV-P1-PCE-2023-0060, within PNCDI IV}

\renewcommand{\iff}{if and only if }

\newcommand{\la}{\longrightarrow}

\newcommand{\N}{\mathbb{N}}
\newcommand{\Z}{\mathbb{Z}}
\newcommand{\Q}{\mathbb{Q}}

\DeclareMathOperator{\Hom}{Hom}
\DeclareMathOperator{\RHom}{{\bf R}Hom}

\DeclareMathOperator{\id}{\bf 1}

\newcommand\lotimes{\otimes^{\mathbf L}}
\newcommand\yon{\mathbf y}

\newcommand{\CL}{\mathcal{L}}

\newcommand{\Y}{\mathcal{Y}}
\newcommand{\C}{\mathcal{C}}
\newcommand{\D}{\mathcal{D}}

\newcommand{\U}{\mathcal{U}}
\newcommand{\CP}{\mathcal{P}}
\newcommand{\X}{\mathcal{X}}
\newcommand{\V}{\mathcal{V}}
\newcommand{\W}{\mathcal{W}}
\newcommand{\CH}{\mathcal{H}}
\newcommand{\G}{\mathcal{G}}
\newcommand{\CS}{\mathcal{S}}
\newcommand{\BS}{\mathbb S}

\newcommand{\Ab}{\mathrm{Ab}}

\newcommand{\Mod}[1]{\hbox{\rm Mod-}{#1}}
\newcommand{\add}{\mathrm{add}}
\newcommand{\Add}{\mathrm{Add}}

\newcommand{\Prod}{\mathrm{Prod}}
\newcommand{\Susp}{\mathrm{Susp}}

\newcommand{\Der}[1]{\mathbf{D}({#1})}

\newcommand{\Loc}{\mathrm{Loc}}

\newcommand{\thick}{\mathrm{thick}\,}

\newcommand{\hocolim}{\mathrm{hocolim}}

\theoremstyle{plain}
\newtheorem{thm}{Theorem}[section]
\newtheorem{lemma}[thm]{Lemma}
\newtheorem{prop}[thm]{Proposition}
\newtheorem{cor}[thm]{Corollary}

\newtheorem{thm-intro}{Theorem}

\theoremstyle{definition}

\theoremstyle{remark}
\newtheorem{rem}[thm]{Remark}
\newtheorem{expl}[thm]{Example}

\newtheorem{setting}[thm]{Setting}

\begin{document}

%\maketitle

\begin{abstract}
Consider an adjoint triple of triangle functors between two nice enough triangulated categories. In this paper, we are looking for conditions under which the silting, respectively, cosilting property ascends or descends via at most left, respectively, at most right adjoint.
\end{abstract}

\maketitle

% ------------------------------------------------------------------------------
\section*{Introduction}

 Let $\varphi_*:\D_2\to\D_1$ be a triangle functor between two nice enough (for simplicity, let say, compactly generated) triangulated categories. Assume that it has a left adjoint $\varphi^*:\D_1\to\D_2$, or a right adjoint $\varphi^!:\D_1\to\D_2$, or sometimes both. Consider two objects $T_1,C_1\in\D_1$ and denote $T_2=\varphi^*T_1\in\D_2$ and $C_2=\varphi^!C_1$. The main question investigated in the present paper is to find validity conditions for the following implications: "$T_1$ is silting $\Rightarrow T_2$ is silting", "$T_2$ is silting $\Rightarrow T_1$ is silting", "$C_1$ is cosilting $\Rightarrow C_2$ is cosilting" or "$C_2$ is cosilting $\Rightarrow C_1$ is cosilting". For the first two, we use the terminology "The silting property ascends, respectively, descends via $\varphi^*$", and for the last two, we will say "The cosilting property ascends, respectively, descends via $\varphi^!$".

 The ascending and descending question for various properties is an old subject of investigation in commutative algebra. The aspect envisioned here is the possibility to transfer the property from the global situation to the local one, and conversely. For example, if $\varphi:A\to B$ is a morphism of commutative rings, then the property "being projective" relative to an $A$-module ascends via induction functor
 $\varphi^*=-\otimes_AB:\Mod A\to\Mod B$. If, moreover, $\varphi$ is faithfully flat, then the same property descends via $\varphi^*$ (see \cite{RG71}). Some notions as (co)silting are properly defined at the derived level. Therefore, we need to replace the tensor functor with the derived version $\varphi^*=-\lotimes_AB:\Der A\to\Der B$. This is the main subject of investigation in \cite{BHM25}. In the present paper, we generalize some results in \cite{BHM25} (see Example \ref{e:derived}, for the precise relationship between the main settings here and there). Instead of the derived category $\Der A$, over a ring $A$, we consider here an arbitrary compactly generated triangulated category $\D$. The commutativity of $A$ is translated here in the hypothesis that $\D$ is a closed tensor triangulated such that the tensor unit $\BS$ is itself compact and silting. Note that this hypothesis is general enough to include not only the derived category over a commutative ring, but also the homotopy category of spectra. We are able to translate, in the new more general setting, all the main results in \cite{BHM25} related to the ascending part. However, the descending part seems to be more complicated and interesting. Even if we found a way to interpret faithfully flatness in the new setting, for the descending of the silting property, we were forced to assume in Theorem \ref{t:silt-desc} and Corollary \ref{c:silt-desc} a stronger hypothesis (compared with the respective result in \cite{BHM25}), namely $T_1$ to be compact.
 Generalizing known results may be interesting by itself, but without telling something new about the original situation may be considered of limited value. Theorem \ref{t:cosilt-desc} shows that this is not the case here. The general approach allowed us to focus on the relevant properties, leading to a proof of the fact that the descending property holds under some mild hypotheses not considered before.

 The paper is organized as follows: In Section \ref{prel} the main notions are defined  and the main settings are established.  In Section \ref{general} some characterizations for the (co)silting property are recalled through some properties that remind us of the definition of (co)tilting objects. Necessarily and sufficient conditions for these properties to migrate via the functors $\varphi^*$ or $\varphi^!$ are also recorded.  The fact that the (co)silting property ascends in tensor triangulated categories is proved in Section \ref{ascending}, more precisely in Theorem \ref{thm:silting-ascends}. Moreover, if $T_1$ and $T_1'$ are equivalent silting objects in $\D_1$, then the induced silting objects $T_2$ and $T_2'$ are equivalent in $\D-2$, and an analogous statement holds for the cosilting case (see Corollary \ref{c:a-silt-are-eq}). In Section \ref{qff} we interpret the faithfully flatness in terms of adjunctions between compactly generated triangulated categories, and we call quasi faithfully flat such an adjunction.  With the hypothesis that the adjunction is quasi faithfully flat we prove in Corollary \ref{c:unique-asc} that if the silting of finite type objects $T_1$ and $T_1'$ in $\D_1$ induces equivalent silting (necessarily of finite type) objects in $\D_2$, then $T_1$ and $T_1'$ are also equivalent. Finally, in Section \ref{descending}, we formulate and prove our versions for the descending results for silting in Corollary \ref{c:silt-desc} and for cosilting in Theorem \ref{t:cosilt-desc}.

\section{Preliminaries and notations}\label{prel}

Let $\D$ be a triangulated category, whose shift functor is denoted by $(-)[1]$. For $n\in\Z$ the $n$-th iteration of the shift functor is denoted by $(-)[n]$. All our subcategories are full therefore, we identify a subcategory with the respective class of objects. If $\X\subseteq\D$ is such a class, we denote $\X[n]=\{X[n]\mid X\in\X\}$, for every $n\in\Z$. Moreover, for $N\subseteq\Z$ we set $\X[N]=\bigcup_{n\in\N}\X[n]$; if $\X=\{X\}$ is a singleton, we suppress unnecessary brackets, and we write simply $X[N]$ for this class. Often we write $[<0]$, $[\leq0]$ etc. instead of $[N]$ for the respective subset of $\Z$.

For $X,Y\in\D$, we say that $X$ is orthogonal to $Y$, provided that $\D(X,Y)=0$. This notion extends easily from objects to classes of objects as follows: If $\X,\Y\subseteq\D$ we write $\D(\X,\Y)$ for the class of all abelian groups $\D(X,Y)$ with $X\in\X$ and all $Y\in\Y$. So we say that $\X$ and $\Y$ are orthogonal if $\D(\X,\Y)=0$.   We define
\[\X^\perp=\{Y\in\D\mid\D(\X,Y)=0\}\hbox{ and }{^\perp\X}=\{Y\in\D\mid\D(Y,\X)=0\}.\]

Let $\X\subseteq\D$ be a subcategory. We call this subcategory: \begin{itemize}
 \item \textit{closed under extensions} if for every triangle
$X\to Y\to Z\to X[1]$ in $\D$ we have $Y\in\X$, whenever $X,Z\in\X$;
\item \textit{(co)smashing} if it is closed under coproducts (products) which exits in $\D$;
\item \textit{(co)suspended} if it is closed under extensions and positive (negative) shifts. If coproducts exist in $\D$ then we can construct $\Susp(\X)$ as the smallest suspended and smashing subcategory of $\D$ containing $\X$.
\end{itemize}

A {\em torsion pair} in $\D$ is a pair $(\U,\V)$ of subcategories $\D$ such that:
\begin{enumerate}
    \item Both $\U$ and $\V$ are closed under direct summands.
    \item $\D(\U,\V)=0$.
    \item Every object $X\in\D$ lies in a triangle $U\to X\to V\to U[1]$ with $U\in\U$ and $V\in\V$.
\end{enumerate}
It is clear that if $(\U,\V)$ is a torsion pair, then $\U={^\perp\V}$ and $\V=\U^\perp$. Consequently $\U$ is closed under coproducts and $\V$ is closed under products that exist in $\D$. Moreover, both of them are closed under extensions. We say that $\U$ is the \textit{aisle}, respectively $\V$ is the \textit{coaisle}\footnote{In the literature the names aisle and coaisle  are used often for t-structures, but we can use them safely in this more general context.} of the torsion pair $(\U,\V)$. Let $\G$ be a set of objects from $\D$. We say that the torsion pair $(\U,\V)$ is \textit{generated by $\G$} if $\V=\G^{\perp}$.

A torsion pair is called \textit{t-structure} (\textit{weight structure}, or briefly \textit{w-structure}) if in addition $\U$ is closed under positive (respectively negative) shifts, that is, if $\U$ is (co)suspended. Note that the closure of $\U$ under positive (negative) shifts, that is, $\U[1]\subseteq\U$ (respectively, $\U[-1]\subseteq\U$) is equivalent to $\V[-1]\subseteq\V$ (respectively, $\V[1]\subseteq\V$). For a t-structure $(\U,\V)$ in $\D$ the triangle
$U\to X\to V\to U[1]$ given in the definition of a torsion pair above is functorial in $X\in\D$ and gives rise to a right adjoint $X\mapsto U$ of the inclusion functor $\U\to\D$ and to a left adjoint $X\mapsto V$ of the inclusion functor $\V\to\D$ (see \cite{BBD82}).
We call this triangle the \textit{t-truncation triangle} associated to $X$. In the case of a w-structure, we still call a triangle as above a \textit{w-truncation triangle} associated to $X\in\D$, but this triangle is not more unique in general (see \cite{B10}). A t-structure (or a w-structure) $(\U,\V)$ is called {\em nondegenerate} if $\bigcap_{n\in\Z}\U[n]=0=\bigcap_{n\in\Z}\V[n]$.

The next result is a slight generalization of a well-known result. Actually, the result is classical provided that the torsion pairs considered below are t-structures, and the (straightforward) proof depends only on the fact that the aisle and the coaisle are left, respectively, right orthogonal to each other.

\begin{lemma}\label{lr-exact} Let $\D_1$ and $\D_2$ be two triangulated categories each of them being endowed a torsion pair $(\U_1,\V_1)$, respectively $(\U_2,\V_2)$. If $\varphi^*:\D_1\rightleftarrows\D_2:\varphi_*$ is an adjoint pair, then $\varphi^*(\U_1)\subseteq\U_2$ \iff $\varphi_*(\V_2)\subseteq\V_1$.
\end{lemma}

In the hypotheses of the above lemma, we say that $\varphi^*$ is {\em right exact} and $\varphi_*$ is {\em left exact} with respect to the respective torsion pairs (sometimes these torsion pairs are tacitly understood), provided that $\varphi^*(\U_1)\subseteq\U_2$, respectively $\varphi_*(\V_2)\subseteq\V_1$.  If the torsion pairs are t-structures (or w-structures) and we use the more classical terminology left/right t-exact (or w-exact). A functor
which is both left and right t-exact (w-exact) will be called simply {\em t-exact} (w-exact).

A TTF triple is a triple $(\U,\V,\W)$ of subcategories of $\D$ such that both $(\U,\V)$ and $(\V,\W)$ are torsion pairs. This TTF triple is called (co)suspended, provided that $\V$ is closed under positive (negative) shifts, that is, provided that $(\U,\V)$ is a w-structure (t-structure) and $(\V,\W)$ is a t-structure (respectively, w-structure).

%We assume that $\D$ has coproducts. An object $T\in\D$ is called {\em silting} if it induces a t-structure $(T[<0]^\perp, T[\geq0]^\perp)$. Dually if $\D$ has products, an object $C\in\D$ is called cosilting if it induces a t-structure $({^\perp C[\leq0]}, )$. This t-structure is called the silting t-structure induced by $T$. Two silting objects are said to be equivalent if they induce the same t-structure. If we assume more, that the ambient triangulated category $\D$ is compactly generated, then silting t-structure extends to a suspended TTF triple
We assume that $\D$ has coproducts. An object $T\in\D$ is called {\em silting} if it induces a t-structure $(T[<0]^\perp, T[\geq0]^\perp)$. Dually, if $\D$ has products, an object $C\in\D$ is called cosilting if it induces a t-structure $({^\perp C[\leq0]},{^\perp C[>0]})$. This t-structure is called the (co)silting t-structure induced by $T$ (respectively $C$). Two (co)silting objects are said to be equivalent if they induce the same t-structure. If we assume that the ambient triangulated category $\D$ is compactly generated, then the (co)silting t-structure extends to a suspended TTF triple
\[\left({^\perp(T[<0]^\perp)}, T[<0]^\perp, T[\geq0]^\perp\right),\hbox{repectively }\left({^\perp C[\leq0]}, {^\perp C[>0]},({^\perp C[.0]})^\perp\right).\]
More precisely, according to \cite[Theorem 4.11]{An-19} there is a bijective correspondence between:
\begin{enumerate}[(1)]
    \item equivalence classes of silting objects;
    \item suspended TTF triples $(\U,\V,\W)$ for which the t-structure $(\V,\W)$ is nondegenerate and $\V=\CS^\perp$ for a  a set (and not a proper class) of objects $\CS\subseteq\D$.
\end{enumerate}

The equivalence between two silting objects can be tested in a simple way:

\begin{lemma}\label{equiv-silting} Let $\D$ be a triangulated category with coproducts and let $T$ and $T'$ be two silting objects in $\D$. Then the following are equivalent:
\begin{enumerate}[{\rm (i)}]
 \item $T$ and $T'$ are equivalent.
 \item $\Add(T)=\Add(T')$.
 \item $T'\in\Add(T)$.
 \item $T\in\Add(T')$.
\end{enumerate}
\end{lemma}

\begin{proof}
The implication (i)$\Rightarrow$(ii) follows from \cite[Theorem 2]{NSZ19}, where $\Add(T)$ is described in terms of the silting t-structure. The implications (ii) $ \Rightarrow $(iii) and (ii)$\Rightarrow$(iv)
are obvious, and we only have to show (iii)$\Rightarrow$(i), since for (iv)$\Rightarrow$(i) we only have to twist $T$ and $T'$.

(iii)$\Rightarrow$(i). Let $(\V,\W)$ and $(\V',\W')$ be the $t$-structures induced by $T$, respectively $T'$. Observe that $T[N]^\perp=\Add(T)[N]^\perp$, for any $N\subseteq\Z$, therefore \[\V=\Add(T)[<0]^\perp\hbox{ and } \W=\Add(T)[\geq0]^\perp,\] and similar for $\V'$, $\W'$ and $T'$.
Condition (iii) implies $\Add(T')\subseteq\Add(T)$, therefore we get:
\[\V=\Add(T)[<0]^\perp\subseteq\Add(T')[<0]^\perp=\V'\]
henceforth $\W'={^\perp\V'}\subseteq{^\perp\V}=\W$. On the other hand, we have:
\[\W=\Add(T)[\geq0]^\perp\subseteq\Add(T')[\geq0]^\perp=\W,\]
which means that $\W=\W'$ and we are done.
\end{proof}

Remark that a dual statement holds for the cosilting case, showing that two cosilting objects $C$ and $C'$ in $\D$ are equivalent if and only if $\Prod(C)=\Prod(C')$, where $\Prod(C)$ denotes the class of all direct summands of products of arbitrary copies of $C$.

\begin{setting}\label{s:tcats} Let $\D_1$ and $\D_2$ two triangulated categories.
\begin{enumerate}
\item[(L)] Suppose that $\D_1$ and $\D_2$ have coproducts, and $\varphi_*:\D_2\to\D_1$ is a triangle functor that has a left adjoint $\varphi^*:\D_1\to\D_2$. Fix an object $T_1\in\D_1$ and denote $T_2=\varphi^*T_1\in\D_2$.
\item[(R)] This is the dual of (1). More precisely, suppose that $\D_1$ and $\D_2$ have products and $\varphi_*:\D_2\to\D_1$ is a triangle functor that has a right adjoint $\varphi^!:\D_1\to\D_2$. Fix an object $C_1\in\D_1$ and denote $C_2=\varphi^!C_1\in\D_2$.
\item[(LR)] This is the conjunction between (1) and (2). More precisely, suppose that $\D_1$ and $\D_2$ have both products and coproducts, and that $\varphi^*\dashv\varphi_*\dashv\varphi^!$ is an adjoint triple of triangle functors. Fix two objects $T_1,C_1\in\D_1$ and denote $T_2=\varphi^*T_1\in\D_2$, $C_2=\varphi^!C_1\in\D_2$.
\end{enumerate}
\end{setting}

Suppose that we are in the hypotheses of Setting \ref{s:tcats} (L). We say that a property {\em ascends via $\varphi^*$}, if $T_2$  satisfies the respective property in $\D_2$, whenever $T_1\in\D_1$ does.  For the converse implication, we say that the property {\em descends via $\varphi^*$}. Examples of properties of interest for us are silting, silting of finite type etc. %Therefore, the silting property ascends (descends) via $\varphi^*$ means the implication "$T_1$ is silting $\Rightarrow T_2$ is silting" (respectively "$T_2$ is silting $\Rightarrow T_1$ is silting") holds.

Analogously, suppose that we are in the hypotheses of Setting \ref{s:tcats} (R). We say that a property, for example, being a cosilting object, {\em ascends via $\varphi^!$}, provided that the following implication holds: If $C_1\in\D_1$ satisfies the respective property, then $C_2\in\D_2$ does.
For the converse implication, we will say that the property descends via $\varphi^!$.

Remark that we can define ascending or descending for a property in a more general case; the categories $\D_1$ and $\D_2$ have not to be triangulated (often they are module categories), they need not to have (co)products etc. But in this paper, we are concerned only with the case mentioned above.

We will use the following result several times. Note that its proof follows straightforwardly from the adjunction:

\begin{lemma}\label{l:phi^-1} \begin{enumerate}[{\rm (1)}]
    \item Assume Setting \ref{s:tcats} (L). Then, for every $N\subseteq\Z$, we have  $$T_2[N]^\perp=\varphi_*^{-1}(T_1[N]^\perp)\hbox{  and consequently }\varphi_*(T_2[N]^\perp)\subseteq T_1[N]^\perp.$$
    \item Assume Setting \ref{s:tcats} (R). Then, for every $N\subseteq\Z$, we have $$^\perp C_2[N]=\varphi_*^{-1}(^\perp C_1[N])\hbox{ and consequently }\varphi_*(^\perp C_2[N])\subseteq ^\perp{C_1[N]}.$$
\end{enumerate}
\end{lemma}

%\begin{proof} For $Y\in\D$ the chain of equivalences: \begin{align*}Y\in T_2[N]^\perp&\Leftrightarrow \D_2(T_2[N],Y)=0\Leftrightarrow\D_1(T_1[N],\varphi_*Y)=0\\ &\Leftrightarrow\varphi_*Y\in T_1[N]^\perp\Leftrightarrow Y\in\varphi_*^{-1}(T_1[N]^\perp) \end{align*} proves our claim. Finally $\varphi_*(T_2[N]^\perp)=\varphi_*(\varphi_*^{-1}(T_1[N]^\perp))\subseteq T_1[N]^\perp$ \end{proof}

%\begin{lemma} Let $\varphi_*:\D_2\rightleftarrows\D_1:\varphi^!$ be a pair of adjoint triangle functors between two triangulated categories. Let $C_1\in\D_1$ and $N\subseteq\Z$. If $C_2=\varphi^!C_1\in\D_2$, then $^\perp{C_2[N]}=\varphi_*(^\perp C_1[N])$. \end{lemma}

Since a triangle functor $\varphi_*:\D_2\to\D_1$ commutes with shifts, we do not need additional brackets in the expression $\varphi_*X[n]$, with $X\in\D_2$, and similar for its adjoints. Note also that $^\perp{C[N]}$ should be read as $^\perp(C[N])$.

\section{Formal characterizations for ascending property}\label{general}

In the sequel, we will employ the following characterization of silting objects:

\begin{prop}\label{char-silt}\cite[Corollary 2.7]{Br24}
 Let $\D$ be a triangulated category with coproducts. An object $T\in\D$ is silting if and only if
    \begin{enumerate}[{\rm S}1.]
        \item $T$ is semi-rigid, that is, $T\in T[<0]^\perp$.
        \item $T[<0]^\perp$ is closed under coproducts.
        \item $T$ generates $\D$, that is, $T[\Z]^\perp=\{0\}$.
    \end{enumerate}
\end{prop}

Remark that the equality $\Loc(T)=\D$ implies condition S3 above; indeed, if $X\in T[\Z]^\perp$, then $X\in\Loc(T)^\perp$. But for the converse implication we need some additional assumptions.

 %, and we denote by
%$$\omega_{A,B}:\C(A,\varphi_*B)\to \D(\varphi^*A,B)$$
%the natural isomorphism induced by this adjoinction.
%It is well--known, and also straightforward to show, that if the left adjoint $\varphi^*$ is right exact with respect to a \iff the right adjoint $\varphi^*$ is left exact. \textcolor{red}{Moreover if this is the case they induce adjoint functors
%\[\X\cup\Y\stackrel{\subseteq}\longrightarrow\C\stackrel{\varphi^*}\longrightarrow\D\stackrel{H^0}\longrightarrow\X\cap\Y\hbox{\rm\ \ and\ \ }\X\cap\Y\stackrel{\subseteq}\longrightarrow\D\stackrel{\varphi_*}\longrightarrow\C\stackrel{H^0}\longrightarrow\X\cap\Y\]
%between the respective hearts. }

\begin{lemma}\label{cond-S}% \label{induced-silt}
 Assume Setting \ref{s:tcats} (L). Then the following statements hold true:
 \begin{enumerate}[{\rm (a)}]
     \item  $T_2$ satisfies S1 \iff $\varphi_*T_2\in T_1[<0]^\perp$.
     \item If $T_1$ satisfies S2 and $\varphi_*$ preserves coproducts, then $T_2$ satisfies S2.
     \item If $T_1$ satisfies S3 and $\varphi_*$ reflects zero objects, then $T_2$ satisfies S3.
     \item If $T_2$ satisfies S3 then $\varphi_*$ reflects zero objects.
 \end{enumerate}
\end{lemma}

\begin{proof} The statement (a) follows by the adjunction isomorphism
\[\D_1(T_1[<0],\varphi_*T_2)\cong\D_2(\varphi^*T_1[<0],T_2)=\D_2(T_2[<0],T_2).\]

(b). Let $(X_i)_{i\in I}$ be a family of objects in $T_2[<0]^\perp$. Then we have
\[\D_1(T_1[<0],\varphi_*X_i)\cong\D_2(\varphi^*T_1[<0],X_i)=\D_2(T_2[<0],X_i)=0,\] hence $\varphi_*X_i\in T_1[<0]^\perp$, for all $i\in I$. Using S2 for $T_1$ and the hypothesis that $\varphi_*$ preserves coproducts, we deduce $\varphi_*(\coprod X_i)\cong\coprod\varphi_*X_i\in T_1[<0]^\perp$. Thus
\[\D_2\left(T_2[<0],\coprod X_i\right)=\D_2\left(\varphi^*T_1[<0],\coprod X_i\right)\cong\D_1\left(T_1[<0],\varphi_*(\coprod X_i)\right)=0,\] which shows S2 for $T_2$.

(c). Let $X\in T_2[\Z]^\perp$. Using again the adjunction isomorphism  \[\D_1(T_1[\Z],\varphi_*X)\cong\D_2(\varphi^*T_1[\Z],X)=\D_2(T_2[\Z],X)=0\]
we deduce $\varphi_*X\in T_1[\Z]^\perp=\{0\}$. Since $\varphi_*$ is supposed to reflect zero objects, it follows that $X=0$.

(d). Suppose that $\varphi_*X=0$. Then  \[\D_2(T_2[\Z],X)=\D_2(\varphi^*T_1[\Z],X)\cong\D_1(T_1[\Z],\varphi_*X)=0,\] hence S3 for $T_2$ implies $X=0$.
\end{proof}

We summarize the above results in the following Proposition. It contains an equivalent condition to the ascending property for silting objects, condition which follows formally from the adjunction property: 

\begin{prop}\label{induced-silt}
In Setting \ref{s:tcats} (L), suppose that $\varphi_*$ preserves coproducts and reflects zero objects, and that $T_1\in\D_1$ be a silting object. The following are equivalent:
\begin{enumerate}[\rm (i)]
    \item $T_2$ is silting in $\D_2$;
    \item $\varphi_*T_2\in T_1[<0]^\perp$.
\end{enumerate}
 Moreover, if the above equivalent conditions hold, then $\varphi_*$ is t-exact and $\varphi^*$ is right t-exact w.r.t. the silting t-structures induced by $T_1$ and $T_2$. %(i.e. $\varphi_*(\V_2)\subseteq\V_1$, $\varphi_*(\W_2)\subseteq\W_1$ and $\varphi^*(\V_1)\subseteq\V_2$, where  $(\V_1,\W_1)$  and $(\V_2,\W_2)$ are the respective silting t-structures).
 %Furthermore, if the categories $\D_1$ and $\D_2$ are compactly generated and $(\U_1,\V_1,\W_1)$ and $(\U_2,\V_2,\W_2)$ are the TTF triples induced by $T_1$ and $T_2$, then $\varphi^*$ is exact w.r.t. the silting w-structures induced by $T_1$ and $T_2$. %(i.e. we have in addition $\varphi^*(\U_1)\subseteq\U_2$).
\end{prop}

\begin{proof}
The equivalence follows from the characterization of a silting object given in Proposition \ref{char-silt} and the transfer of the properties S1-S3 from $T_1$ to $T_2=\varphi^*T_1$ established in Lemma \ref{cond-S}.

Further, the inclusions $\varphi_*(T_2[<0]^\perp)\subseteq T_1[<0]^\perp$
and $\varphi_*(T_2[\geq0]^\perp)\subseteq T_1[\geq0]^\perp$ follow from Lemma \ref{l:phi^-1}. Finally, the right t-exactness of $\varphi^*$ follows using Lemma \ref{lr-exact} and the above established t-exactness of $\varphi_*$.
\end{proof}

Note that everything in this section can be dualized in order to get results about cosilting objects. For the convenience of the reader, we formulate the dual results that we will apply in the sequel.

\begin{prop}\label{char-cosilt}\cite[Corollary 2.9]{Br24}
 Let $\D$ be a triangulated category with products. An object $C\in\D$ is cosilting if and only if
    \begin{enumerate}[{\rm C}1.]
        \item $C$ is semi-corigid, that is, $C\in{^\perp C[>0]}$.
        \item ${^\perp C[>0]}$ is closed under products.
        \item $C$ cogenerates $\D$, that is, ${^\perp C[\Z]}=\{0\}$.
    \end{enumerate}
\end{prop}

\begin{lemma}\label{cond-C}% \label{induced-silt}
 Assume Setting \ref{s:tcats} (R). Then the following statements hold true:
 \begin{enumerate}[{\rm (a)}]
     \item  $C_2$ satisfies C1 \iff $\varphi_*C_2\in{^\perp C_1[>0]}$.
     \item If $C_1$ satisfies C2 and $\varphi_*$ preserves products, then $C_2$ satisfies C2.
     \item If $C_1$ satisfies C3 and $\varphi_*$ reflects zero objects, then $C_2$ satisfies C3.
     \item If $C_2$ satisfies C3 then $\varphi_*$ reflects zero objects.
 \end{enumerate}
\end{lemma}

\begin{prop}\label{induced-cosilt}
In Setting \ref{s:tcats} (R), suppose that $\varphi_*$ preserves products and reflects zero objects, and that $C_1\in\D_1$ be a cosilting object. The following are equivalent:
\begin{enumerate}[\rm (i)]
    \item $C_2$ is cosilting in $\D_2$;
    \item $\varphi_*C_2\in{^\perp C_1[>0]}$.
\end{enumerate}
Moreover, if the above equivalent conditions hold, then $\varphi_*$ is t-exact and $\varphi^!$ is left t-exact w.r.t. the cosilting t-structures induced by $C_1$ and $C_2$.
\end{prop}

%In the hypotheses of Proposition \ref{induced-silt} before, we say that the property of being silting ascends along $\varphi^*$, provided that $T_1\in\D_1$ is silting implies $T_2=\varphi^*T_1\in\D_2$ is silting.

\section{Ascending for the (co)silting property in tensor triangulated categories}\label{ascending}

Recall that a {\em symmetric monoidal category} is a category $\D$ endowed with a bifunctor
$-\otimes-:\D\times\D\to\D$ called {\em tensor product}, and an object $\BS\in\D$ such that the tensor product is associative, commutative and has $\BS$ as neutral element up to natural isomorphisms, isomorphisms which have to satisfy some compatibility conditions, see \cite[Chapter VII, Section 7]{MacL71}. Such a category is called {\em closed} if for every $X\in\D$ the functor $-\otimes X:\D\to\D$ has a right adjoint $\hom(X,-):\D\to\D$; this adjoint is called the {\em internal hom functor}. Note that if we fix the second variable in the internal hom functor, we get a contravariant functor $\hom(-,Y):\D^\mathrm{op}\to\D$ which sends coproducts into products. A {\em tensor-triangulated category} is a triangulated category endowed with a symmetric monoidal structure such that the tensor product is triangulated in both variables.  If, in addition, $\D$ is closed, then $\hom(X,-)$ and $\hom(-,Y)$ are triangle functors for all $X,Y\in\D$. Note also that the adjunction isomorphism $\D(X\otimes Y,Z)\cong\D(X,\hom(Y,Z))$ has an enriched version:
\[\hom(X\otimes Y,Z)\cong\hom(X,\hom(Y,Z)),\] natural in all three variables. Indeed, this follows from the Yoneda lemma and the chain of natural isomorphisms:
\begin{align*}
    \D(-,\hom(X\otimes Y,Z))&\cong\D(-\otimes X\otimes Y,Z)\cong\D(-\otimes X,\hom(Y,Z)) \\
    &\cong\D(-,\hom(X,\hom(Y,Z))).
\end{align*}

\begin{lemma}\label{lem:aisle-in-tt}
Let $(\D, \otimes, \BS)$ be a closed tensor-triangulated category with coproducts. Assume that $\BS$ is silting, consider an arbitrary t-structure $(\V,\W)$ in $\D$, and an object $X\in\BS[<0]^\perp$.
\begin{itemize}
 \item[\rm (a)] For every $V\in\V$  we have $V\otimes X\in\V$.
 \item[\rm (b)] For every $W\in\W$ we have $\hom(X,W)\in\W$.
\end{itemize}
\end{lemma}

\begin{proof} Note that if $\D$ is a closed tensor-triangulated category, then for every $Y\in\D$ the (covariant) functor $-\otimes Y$ is triangulated and preserves coproducts, therefore it preserves homotopy colimits too. The (contravariant) functor $\hom(-,Y)$ is triangulated and sends coproducts into products therefore it converts homotopy colimits into homotopy limits.

Since $X\in\BS[<0]^\perp$ it follows by \cite[Theorem 2]{NSZ19} (see also \cite[Proposition 4.8]{An-19} or \cite[Section 4]{PS18}) that $X\cong\hocolim_{n\geq0}X_n$ where
\[X_{-1}\stackrel{\xi_{-1}}\to X_0\stackrel{\xi_0}\to X_1\stackrel{\xi_1}\to X_2\stackrel{\xi_2}\to\ldots \]
is a sequence such that $X_{-1}=0$ and for all $n\geq0$ there is a triangle
\[X_{n-1}\stackrel{\xi_{n-1}}\to X_n\to P_n\to X_{n-1}[1]\]
with $P_n\in\Add(\BS[n])$. Clearly, $X_0\cong P_0$.

(a) Since $V\oplus\BS\cong V$ belongs to $\V$, the tensor product commutes with coproducts and $\V$ is smashing, closed under summands and positive shifts, we obtain $V\otimes P_n\in\V$ for all $n\geq0$. Since $\V$ is closed under extensions as well, an easy inductive argument shows that $V\otimes X_n\in\V$, for all $n\geq0$. But $\V$ is closed under homotopy colimits, thus \[V\otimes X\cong V\otimes\hocolim_{n\geq0}X_n\cong\hocolim_{n\geq0}(V\otimes X_n)\in\V  .\]

(b) is the dual of (a).
\end{proof}

Let $(\D_1,\otimes,\BS_1)$ and $(\D_2,\otimes,\BS_2)$ be two symmetric monoidal categories. Recall that a functor $\varphi^*:\D_1\to\D_2$ is called {\em lax monoidal} if it comes togheter with a natural morphism $\varphi^*(-)\otimes\varphi^*(-)\to\varphi^*(-\otimes-)$ and a morphism $\BS_2\to\varphi^*\BS_1$. If in addition these morphisms are supposed to be isomorphisms, then the functor is called {\em strong monoidal} (see \cite[p. 255]{MacL71}). An adjunction $\varphi^*:\D_1\rightleftarrows\D_2:\varphi_*$ between them is called {\em monoidal} if $\varphi^*$ is strong monoidal, case in which $\varphi_*$ is automatically lax monoidal, by \cite[Theorem 1.5]{Ke74}. 

\begin{setting}\label{s:ttcats} Let $(\D_1,\otimes,\BS_1)$ and $(\D_2,\otimes,\BS_2)$ be two closed tensor-triangulated categories.
Suppose that both $\D_1$ and $\D_2$ have products and coproducts, and consider an adjoint triple of triangulated functors $\varphi^*\dashv\varphi_*\dashv\varphi^!$, with the property that $\varphi^*:\D_1\rightleftarrows\D_2:\varphi_*$ is a  monoidal adjunction.
\end{setting}

\begin{expl}\label{e:derived}
Let $A$ and $B$ be two commutative rings, and let $\varphi:A\to B$ be a unital ring homomorphism. It induces a (derived) restriction functor \[\varphi_*:\Der B\to\Der A\hbox{, given by }\varphi_*Y=\RHom_B(B,Y)\cong Y\lotimes_BB.\]  This functor has both a left and a right adjoint, namely the (derived) induction and coinduction functors \[\varphi^*,\varphi^!:\Der A\to\Der B\hbox{, where }\varphi^*X=X\lotimes_AB\hbox{, and }\varphi^!X=\RHom_A(B,X).\]
Then $\Der A$ and $\Der B$ are compactly generated. Since the rings are commutative, it follows that the derived tensor product induces a tensor triangulated structure on the derived categories.  Moreover, they are closed, the internal hom functor being $\RHom$. Finally, it is clear that $\varphi^*=-\lotimes_AB$ is strong monoidal, therefore all hypotheses of Setting \ref{s:ttcats} are satisfied.
\end{expl}

%\begin{expl}\label{e:hopf} This is rather a non-example than an example. Let $A$ be a an algebra over a commutative ring $K$. If $A$ is not commutative, then we can form the so called enveloping algebra $A^e=\A\op\otimes_KA$, so we get an internal tensor product $-\otimes_K-:\Mod{A^e}\times\Mod{A^e}\to\Mod{A^e}$.  
%\end{expl}

The following is a sort of Eilenberg-Watts theorem for monoidal adjunctions between tensor-triangulated categories.

\begin{prop}\label{prop:EW}
 Assume Setting \ref{s:ttcats}, and suppose further that $\D_1=\Loc(\BS_1)$. Then there is an isomorphisms
 $$X\otimes\varphi_*\BS_2\to\varphi_*\varphi^*X\hbox{ and }\hom(\varphi_*\BS_2,X)\to\varphi_*\varphi^! X,$$ that are natural in $X\in\D_1$.
\end{prop}

\begin{proof}
 Apply the functor $-\otimes\varphi\BS_2$ to the unit $\id_{\D_1}\to\varphi_*\varphi^*$ of the adjunction between $\varphi^*$ and $\varphi_*$, then use the natural morphism coming from the fact that $\varphi_*$ is lax monoidal and the neutrality of $\BS_2$ with respect to the tensor product in $\D_2$ in order to get the composite natural transformation between functors $\D_1\to\D_1$:
 \[(-)\otimes\varphi_*\BS_2\to\varphi_*\varphi^*(-)\otimes\varphi_*\BS_2\to\varphi_*(\varphi^*(-)\otimes\BS_2)\stackrel{\cong}\to\varphi_*\varphi^*(-).\]
 But $\varphi^*$ is strong, therefore evaluating this natural morphism at $\BS_1$ we get an isomorphism $\BS_1\otimes\varphi_*\BS_2\cong\varphi_*\BS_2\cong\varphi_*\varphi^*\BS_1$. Since both functors $-\otimes\varphi_*\BS_2$ and $\varphi_*\varphi^*$ are triangulated and coproduct preserving, the morphism  is actually an isomorphism evaluated at any $X\in\Loc(\BS_1)=\D_1$.

 For the second isomorphism, we only have to note that $\varphi_*\varphi^!$ is the right adjoint of $\varphi_*\varphi^*$, and to use the first one.
\end{proof}

%\begin{cor}\label{c:coEW}
%Assume Setting \ref{s:ttcats}, and suppose further that $\D_1=\Loc(\BS_1)$. Then there is an isomorphism $\hom(\varphi_*\BS_2,X)\to\varphi_*\varphi^! X$, natural in $X\in\D_1$.
%\end{cor}

%\begin{proof}
 %   We only have to note that $\varphi_*\varphi^!$ is the right adjoint of $\varphi_*\varphi^*$, and apply Proposition \ref{prop:EW}.
%\end{proof}

\begin{thm}\label{thm:silting-ascends} %\textcolor{red}{Ceva ipoteze ca sa se poata aplica \ref{induced-silt} si \ref{prop:EW}.}
In Setting \ref{s:ttcats}, suppose in addition that both tensor units $\BS_1$ and $\BS_2$ are silting objects, and $\D_1=\Loc(\BS_1)$. Then the following statements hold:
\begin{enumerate}[\rm (1)]
 \item The silting property ascends via $\varphi^*$.
 \item The cosilting property ascends via $\varphi^!$.
\end{enumerate}
\end{thm}

\begin{proof}
In a monoidal adjunction, the left adjoint $\varphi^*$ is strongly monoidal, therefore $\varphi^*X\otimes\varphi^*Y\cong\varphi^*(X\otimes Y)$, for all $X,Y\in\D_1$ and $\BS_2\cong\varphi^*\BS_1$. Since S3 holds for $\BS_2$, we know by Lemma \ref{cond-S} (d), that $\varphi_*$ reflects zero objects.

(1). Let $T_1$ be a silting object in $\D_1$ and denote $T_2=\varphi^*T_1\in\D_2$. By the considerations above, the hypotheses of Proposition \ref{prop:EW} are satisfied, therefore it provides an isomorphism \[\varphi_*T_2=\varphi_*\varphi^*T_1\cong T_1\otimes\varphi_*\BS_2.\]
By hypothesis $\BS_2\in\D_2$ is silting, hence the equivalent conditions from Proposition \ref{induced-silt} imply that $\varphi_*\BS_2\in\BS_1[<0]^\perp$. Since $T_1$ belongs to the aisle $T_1[<0]^\perp$ of the respective silting t-structure, Lemma \ref{lem:aisle-in-tt} (a) together with the above isomorphism implies $\varphi_*T_2\in T_1[<0]^\perp$, and the conclusion follows applying Proposition \ref{induced-silt}.

(2). Let $C_1$ be a cosilting object in $\D_1$, and denote $C_2=\varphi^!C_1\in\D_2$. The proof uses the same ideas as in the case (1), with the difference that we need the second isomorphism in Proposition \ref{prop:EW} in order to get an isomorphism
\[\varphi_*C_2=\varphi_*\varphi^!C_1\cong \hom(\varphi_*\BS_2,C_1).\]
Then we apply Lemma \ref{lem:aisle-in-tt} (b) and Proposition \ref{induced-cosilt}.
\end{proof}

An immediate application of Lemma \ref{l:phi^-1} gives:

\begin{cor}\label{c:a-silt-are-eq} Assume the hypotheses of Theorem \ref{thm:silting-ascends}.
\begin{enumerate}[\rm (1)]
 \item If $T_1,T_1'\in\D_1$ are two equivalent silting objects, then the induced silting objects $T_2=\varphi^*T_1,T_2'=\varphi^*T_1'\in\D_2$ are also equivalent.

 \item If $C_1,C_1'\in\D_1$ are two equivalent cosilting objects, then the induced cosilting objects $C_2=\varphi^!C_1, C_2'=\varphi^!C_1'\in\D_2$ are also equivalent.
\end{enumerate}
\end{cor}

\section{Quasi faithfully flat adjunction}\label{qff}

In what follows, we need some additional terminology. Recall that an object $P$ of a triangulated category $\D$ is called {\em compact} if the functor $\D(P,-)$ preserves coproducts. We write $\D^c$ for the full subcategory of $\D$ consisting of all compact objects. Furthermore, the category $\D$ is called {\em compactly generated} if $\D^c$ is essentially small and generates $\D$ (in the triangulated sense), that is, ${\D^c}^\perp=\{0\}$.

Note that if $\D$ is compactly generated, then both $\D$ and $\D^{\mathrm{op}}$ satisfy Brown representability. It follows $\Loc(T)=\D$ for any object $T$ satisfying $T[\Z]^\perp=\{0\}$, in particular for any silting object. Another consequence is that $\D$ also has products.

%{\color{red}{Purity in compactly generated categories}}

Let $\D$ be a compactly generated triangulated category and denote by $\D^c$ its subcategory consisting of all compact objects. By assumption, $\D^c$ is essentially small, and the category of all additive (contravariant) functors $\D^c\to\Ab$ lives in the same universe as $\D$. We call a {\em (right) $\D^c$-module} a functor as before and, we denote by $\Mod{\D^c}$ the category of all such functors. We consider the restricted Yoneda functor $\yon:\D\to\Mod{\D^c}$, $\yon(X)=\D(-,X)|_{\D^c}$. A triangle $X\stackrel{i}\to Y\stackrel{p}\to Z\stackrel{f}\to X[1]$ in $\D$ is said to be {\em pure} if the induced sequence $0\to\yon(X)\to\yon(Y)\to\yon(Z)\to0$ is exact in
$\Mod{\D^c}$. In this case, we say that $i$ is a pure monomorphism,  $p$ is a pure
epimorphism, $f$ is a phantom map. Furthermore, $X$ is a pure subobject and $Z$ a pure quotient of $Y$.  It is clear that a map $f$ is phantom \iff $\yon(f)=0$
An object $D\in\D$ is called {\em pure projective} if it is projective with respect to pure exact triangles, that is, the sequence \[0\to\D(D,X)\to\D(D,Y)\to\D(D,Z)\to0\] is exact for every pure triangle $X\to Y\to Z\to X[1]$. Dually we define {\em pure injective} objects.
We recall that an object $P\in\D$ is pure projective exactly if $\yon(P)$ is projective in $\Mod{\D^c}$. Another characterization for the same property is that $P\in\Add(\D^c)$. (see \cite[Lemma 8.1 and Proposition 8.4]{B00}). For the dual notion hold $D\in\D$ to be pure injective \iff $\yon(D)$ is injective in $\Mod{\D^c}$ \iff the summation map $D^{(I)}\to D$  factors through the canonical map $D^{(I)}\to D^I$. (see \cite[Theorem 1.8]{K00}). The last condition can be formulated as follows: for each set $I$, there is a morphism
$f:D^I\to D$ such that $f\circ \rho_i = 1_D$, where $\rho_i:D
\to D^I$ is the canonical section, for each $\in I$ (see \cite[Definition 5.1]{SS23}).

In the rest of this section we made the assumptions of Setting \ref{s:tcats} (LR) and, in addition, we suppose that $\D_1$ and $\D_2$ are compactly generated.

\begin{prop}
The functor $\varphi^*$ preserves pure projective objects, and the functor $\varphi^!$ preserves pure injective objects.
\end{prop}

\begin{proof}
    Since $\varphi_*$ is exact and preserves coroducts, it follows that $\varphi^*$ preserves compact objects. Further, it preserves directs summands of coproducts of compact objects, that is, pure projective objects.

The functor $\varphi^!$ preserves products, therefore it preserves pure injective objects, by \cite[Lemma 5.3]{SS23}.
\end{proof}

We will say that the adjunction $\varphi^*\dashv\varphi_*$ is  {\em quasi faithfully flat}, provided that the functor $\varphi^*:\D_1\to\D_2$ reflects the phantom maps, that is, $f:X\to X'$ in $\D_1$ is phantom, whenever $\varphi^*f:\varphi^*X\to\varphi^*X'$ is phantom in $\D_2$. The terminology comes from the fact that in in Example \ref{e:derived}, if the morphism of rings  $\varphi:A\to B$ is faithfully flat, then \cite[Proposition 2.7]{BHM25} says that the derived induction and derived restriction functors form a quasi faithfully flat adjoint pair.

A TTF-triple $(\U,\V,\W)$ is called {\em compactly generated} if there is a set of compact objects $\CP\subseteq\D^c$ such that $\V=\CP^\perp$. 
A (co)silting object is called {\em of (co)finite type} if the (co)silting TTF-triple is compactly generated.

\begin{prop}
If the silting property ascends via $\varphi^*$, then the property "silting of finite type" ascends via $\varphi^*$.
If the cosilting property ascends via $\varphi^!$, then the property "cosilting of cofinite type" ascends via $\varphi^!$.
\end{prop}

\begin{proof}
Let $T_1\in\D_1$ be a silting object of finite type, and let $T_2=\varphi^*T_1$. By the hypothesis $T_2$ is silting. Let $\CP\subseteq\D_1^c$ be a set of compact objects such that $\CP^\perp= T_1[<0]^\perp$. Since $\varphi^*$ preserves compacts, we infer that $\varphi^*(\CP)\subseteq\D_2^c$.  For an object $Y\in\D_2$ we have the the sequence of equivalent statements (the first of them comes from Lemma \ref{l:phi^-1}):
\begin{align*}
    Y\in T_2[<0]^\perp&\Leftrightarrow \varphi_*Y\in T_1[<0]^\perp\Leftrightarrow \varphi_*Y\in\CP^\perp\Leftrightarrow \D_1(\CP,\varphi_*Y)=0\\
    &\Leftrightarrow \D_2(\varphi^*(\CP),Y)=0\Leftrightarrow  Y\in\varphi^*(\CP)^\perp
\end{align*}
showing that $T_2[<0]^\perp=\varphi^*(\CP)^\perp$, so $T_2$ is of finite type.

For cosilting objects of cofinite type, a similar argument works.
\end{proof}

\begin{lemma}\label{l:pure-tri}
If $\varphi^*\dashv\varphi_*$ is a quasi faithfully flat adjunction, then for every $X\in\D_1$ the triangle obtained via the completion of the unit of the adjunction:
\[X\to\varphi_*\varphi^*X\to X'\to X[1] \] is pure in $\D_1$.
\end{lemma}

\begin{proof}
The conclusion follows immediately from the assumption on $\varphi^*\dashv\varphi_*$, since $\varphi^*$ sends the triangle above to a split triangle in $\D_2$.
\end{proof}

\begin{prop}  Suppose that $\varphi^*\dashv\varphi_*$ is a quasi faithfully flat adjunction.
\begin{enumerate}[(1)]
\item If the silting property ascends via $\varphi^*$ and $T_1\in\D_1$ is silting of finite type, then $T_1[<0]^\perp=(\varphi^*)^{-1}(T_2[<0]^\perp)$.
 \item If the cosilting property ascends via $\varphi^!$ and $C_1$ is pure injective cosilting, then   $^\perp{C_1[>0]}=(\varphi^*)^{-1}(^\perp{C_2[>0]})$.
\end{enumerate}
\end{prop}

\begin{proof}
We will prove (2) since in the cosilting case is somehow surprising that we  have to deal with $\varphi^*$ instead of $\varphi^!$. For the silting case, the argument is similar.

For simplicity we will denote by $(\U_i,\V_i)=\left({^\perp C_i[\leq 0]},{^\perp C_i[>0]}\right)$, $i=1,2$,  the respective cosilting t-structures. By Lemma \ref{l:phi^-1}, we know that $\varphi_*(\U_2)\subseteq\U_1$ and $\varphi_*(\V_2)\subseteq\V_1$. Then $\varphi^*(\U_1)\subseteq\U_2$ and $\varphi^!(\V_1)\subseteq\V_2$, see Lemma \ref{lr-exact}. From  $\varphi^*(\U_1)\subseteq\U_2$ follows $\U_1\subseteq(\varphi^*)^{-1}(\U_2)$. Conversely, if $X\in(\varphi^*)^{-1}(\U_2)$, then \[\D_1(\varphi_*\varphi^*X,Y)\cong\D_2(\varphi^*X,\varphi^!Y)=0,\] for all $Y\in\V_1$, therefore $\varphi_*\varphi^*X\in{^\perp\V_1}=\U_1$. Since $C_1$ is pure injective, $\U_1$ is closed under pure subobjects, therefore $X\in\U_1$ according to Lemma \ref{l:pure-tri}. 

%(1). From the last part of Proposition \ref{induced-silt}, we know that $\varphi^*(T_1[<0]^\perp)\subseteq T_2[<0]^\perp$, therefore $T_1[<0]^\perp\subseteq(\varphi^*)^{-1}(T_2[<0]^\perp)$. Conversely, if $X\in\D_1$ is such that $\varphi^*X\in T_2[<0]^\perp$, then \[\D_1(T_1[<0],\varphi_*\varphi^*X)\cong\D_2(T_2[<0],\varphi^*X)=0,\] so $\varphi_*\varphi^*X\in T_1[<0]^\perp$. By Lemma \ref{l:pure-tri} the map $X\to\varphi_*\varphi^*X$ is a pure monomorphism. But $T_1[<0]^\perp$ is closed under pure subobjects, since it is generated by compact objects, so $X\in T_1[<0]^\perp$.

%(2) For the cosilting case, the dual argument works, observing that $^\perp{C_1[>0]}$ is closed under pure subobjects, provided that $C_1$ is pure injective.
\end{proof}

\begin{cor}\label{c:unique-asc} Suppose that $\varphi^*\dashv\varphi_*$ is a quasi faithfully flat adjunction.
\begin{enumerate}[(1)]
\item If the silting property ascends via $\varphi^*$ and $T_1$ and $T_1'$ are silting of finite type in $\D_1$, such that the silting objects $T_2=\varphi^*T_1$
and $T_2'=\varphi^*T_1'$ are equivalent in $\D_2$ then $T_1$ and $T_1'$ are also equivalent.
\item If the cosilting property ascends via $\varphi^!$ and $C_1$ and $C_1'$ are pure injective cosilting in $\D_1$, such that the cosilting objects $C_2=\varphi^!C_1$
and $C_2'=\varphi^!C_1'$ are equivalent in $\D_2$ then $C_1$ and $C_1'$ are also equivalent.
\end{enumerate}
\end{cor}

%\begin{prop} Suppose $\varphi^*:\D_1\to\D_2$ preserves the phantom maps. If $T_1'\in\D_1$ is such that $\varphi^*T_1'=T_2$ is pure projective, then there is $T_1\in\D_1$ pure projective, such that $\varphi^*T_1=T_2$. If, moreover, $T_2$ is compact, then $T_1$ can also be chosen to be compact.\end{prop}

%\begin{proof}
 %   There is a pure exat triangle $T_1''\to P_1\to T_1'\stackrel{h}\to T_1''[1]$ with pure projective $T_1$ and phantom $h$. (Actually $T_1$ is a pure projective precover of $T_1'$.) Applying $\varphi^*$ we get a pure exact triangle $\varphi$ ????
%\end{proof}

%\begin{rem}
 %Note that cosilting objects of cofinite type are automatically pure injective, but the converse does not hold in general. However, there are known some cases where the two notions coincide, e. g.????

 %in some particular cases we know that pure injective cosilting objects are automatically of cofinite type.
%\end{rem}

\section{Descending for (co)silting property}\label{descending}

\begin{thm}\label{t:silt-desc} In the hypotheses of Setting \ref{s:tcats} (L), suppose the categories $\D_1$ and $\D_2$ to be compactly generated, $\varphi_*$ preserves coproducts, the adjunction $\varphi^*\dashv\varphi_*$ is quasi faithfully flat.  Suppose that $T_1$ is a generator for $\D_1$ and $T_1[<0]^\perp$ is closed under pure subobjects and coproducts. If $T_2$ is a silting object in $\D_2$, then $T_1$ is silting in $\D_1$. In particular, the hypotheses on $T_1$ are satisfied, provided that $T_1$ is a compact generator of $\D_1$; in this case $T_2$ is also a compact generator of $D_2$.
\end{thm}

\begin{proof}
    Since $T_1$ is already supposed to satisfy conditions S2 and S3 from Proposition \ref{char-silt}, we need to check only S1. We know by Lemma \ref{cond-S} (a), that $\varphi_*\varphi^*T_1\in T_1[<0]^\perp$. According to Lemma \ref{l:pure-tri}, $T_1$ is a pure subobject of $\varphi_*\varphi^*T_1$, and the closure of $T_1[<0]^\perp$ under pure subobjects leads to the desired conclusion.

    Finally, if $T_1$ is compact, then $T_1[<0]^\perp$ is obviously closed under pure subobjects and coproducts, and $T_2=\varphi^*T_1$ is compact in $\D_2$.
\end{proof}

%\begin{rem}  With the terminology in \cite{AMV20}, the conclusion of Theorem \ref{t:silt-desc} can be reformulated as follows: $T_1$ is partial silting in $\D_1$ (see \cite[Proposition 3.5]{AMV20}).
%\end{rem}

\begin{cor}\label{c:silt-desc} In the hypotheses of Setting \ref{s:ttcats}, suppose that $\BS_2$ is compact in $\D_2$, the adjunction $\varphi^*\dashv\varphi_*$ is quasi faithfully flat, and $\Loc(\varphi_*\BS_2)=\D_1$. If $T_1$ is compact in $\D_1$ and $T_2$ is a silting (necessarily compact) object in $\D_2$, then $T_1$ is also silting.
\end{cor}

\begin{proof}
 In order to apply Theorem \ref{t:silt-desc} we only have to show that $T_1$ is a generator for $\D_1$, or equivalently $\Loc(T_1)=\D_1$.
 By hypothesis, $T_2$ is a compact generator of $\D_2$, therefore, $\D_2^c=\thick(T_2)$. As for any other object of $\D_2^c$, we know that $\BS_2\in\add(T_2)[-l]*\ldots*\add(T_2)[l]$, for some $l\geq0$. Therefore, \[\varphi_*\BS_2\in\add(\varphi_*T_2)[-l]*\ldots*\add(\varphi_*T_2)[l]=\add(T_1\otimes\varphi_*\BS_2)[-l]*\ldots*\add(T_1\otimes\varphi_*\BS_2)[l].\]
We conclude this paragraph, recording the inclusion $\Loc(\varphi_*\BS_2)\subseteq\Loc(T_1\otimes\varphi_*\BS_2)$.

 As we noted in the proof of Theorem \ref{thm:silting-ascends}, $\varphi_*\BS_2\in\BS_1[<0]^\perp$, therefore $\varphi_*\BS_2=\hocolim_{n\geq0}X_n$, where $X_{n-1}\to\X_n\to P_n\to X_{n-1}[1]$ is a triangle with $X_{-1}=0$, and $P_n\in\Add(\BS_1[n])$, for all $n\geq0$ (see \cite[Theorem 2]{NSZ19}; the argument was already used in the proof of Lemma \ref{lem:aisle-in-tt}). Since the tensor functor is triangulated and commutes with coproducts, we obtain $T_1\otimes\varphi_*\BS_2\cong\hocolim_{n\geq0}(T_1\otimes X_n)$.
 Inductively, we get $T_1\otimes X_n\in\Loc(T_1)$, and therefore $T_1\otimes\varphi_*\BS_2\in\Loc(T_1)$. This implies further $\Loc(T_1\otimes\varphi_*\BS_2)\subseteq\Loc(T_1)$. The chain of inclusions
 \[\D_1=\Loc(\varphi_*\BS_2)\subseteq\Loc(T_1\otimes\varphi_*\BS_2)\subseteq\Loc(T_1)\subseteq\D_1\] shows the desired equality.
\end{proof}

Corrollary \ref{c:silt-desc} should be compared with \cite[Proposition 6.4]{BHM25}. With the notations made in Example \ref{e:derived}, this Proposition says that if $T_1$ is isomorphic to a bounded complex of projectives such that $T_2$ is silting, and $\Loc(B)=\Der A$,  then $T_1$ is silting. Our hypothesis $T_1$ is compact is stronger than $T_1$ is isomorphic to a bounded complex of projectives. We were forced to assume this stronger hypothesis because in the general settings we worked we did not find a translation of the following fact: If $T_1\in\Der A$ is isomorphic to a bounded complex projectives, then it can be interpreted as a module over an appropriate $A$-algebra such that $T_1$ is silting if and only if the respective module is tilting (see \cite[6.1]{BHM25}).

Let $(\D,\otimes,\BS)$ be a compactly generated closed tensor-triangulated category. For any two objects $X,Y\in\D$ there is morphism $\hom(X,\BS)\otimes Y\to\hom(X,Y)$, natural in both variables. The object $X$ is called {\em rigid} if this morphism is an isomorphism for all $Y\in\D$.  (see \cite[Definition 1.3]{S16}). The category $\D$ is called {\em rigidly-compactly generated}, provides that the rigid and compact objects coincide. It is clear that in such categories the tensor unit $\BS$ must be compact since it is rigid. Note that if $A$ is a commutative ring, as in Example \ref{e:derived}, then the derived category $\Der A$ is rigidly-compactly generated (see also \cite[Example 2.7]{S16}).

Let $\D$ be a rigidly-compactly generated closed tensor-triangulated category. Following \cite{BW24}, we construct a (contravariant) functor $(-)^+:\D^\mathrm{op}\to\D$, called duality, as follows: For every object $X\in\D$ the functor $\Hom_\Z(\D(\BS,X\otimes-),\Q/\Z):\D^\mathrm{op}\to\Mod{\Z}$ is cohomological and sends coproducts to products. By Brown representability theorem, it must be representable. The object representing it is unique up to a natural isomorphism and will be denoted by $X^+$.  Explicitly, we have: \[\Hom_\Z(\D(\BS,X\otimes-),\Q/\Z)\cong\D(-,X^+).\]
From the very definition of the functor $(-)^+$, it follows that it converts coproducts into products. According to \cite[Proposition 3.13]{BW24}, the triple $(\D,\D,(-)^+)$ is a duality triple in the sense of \cite[Definition 3.4]{BW24}. In particular, $D^+$ is pure injective for any $D\in\D$, and for any $X\in\D$ there is a natural map $i_X:X\to X^{++}$ which is a pure monomorphism. Using these properties, the proof of the following lemma is straightforward.

\begin{lemma}\label{l:pi+} Let $\D$ be a rigidly-compactly generated closed tensor-triangulated category. The following are equivalent for an object $X\in\D$:
\begin{enumerate}[\rm (i)]
    \item $X\in\D$ is pure injective.
    \item The canonical map $i_X:X\to X^{++}$ splits.
    \item $X$ is a direct summand in an object of the form $D^+$, for some $D\in\D$.
\end{enumerate}
\end{lemma}

\begin{lemma}\label{l:x+} Let $\D$ be a rigidly-compactly generated closed tensor-triangulated category. Then we have the following isomorphisms, natural in all variables:
\begin{enumerate}
    \item $X^+\cong\hom(X,\BS^+)$.
    \item $\D(X,Y^+)\cong\D(Y,X^+)$, that is $(-)^+:\D\to\D$ is a contravariant self-adjoint.
    \item $\hom(X,Y^+)\cong(X\otimes Y)^+\cong\hom(Y^+,X)$, natural in $X$ and $Y$.
\end{enumerate}
\end{lemma}

\begin{proof} (1) is (contained in) \cite[Lemma 3.2]{BW24} (ii).

 (2)  Using (1) and the adjunction isomorphism we have:
    \[ \D(X,Y^+)\cong\D(X,\hom(Y,\BS^+))\cong\D(X\otimes Y,\BS^+)\cong\D(Y\otimes X,\BS^+)\cong\D(Y,X^+).\]

(3) Use the same argument as for (2), but with the enriched version of the adjunction isomorphism.
\end{proof}

\begin{prop}\label{p:iso+}  In Setting \ref{s:ttcats}, assume that $\D_1$ and $\D_2$ are rigidly-compactly generated closed tensor-triangulated categories. Then:
\begin{enumerate}
    \item There is a natural isomorphism $(\varphi_*\varphi^*X)^+\cong\varphi_*\varphi^!(X^+),$ for all $X\in\D_1$
    \item Assume also that $\varphi_*(Y^+)\cong(\varphi_*Y)^+$, for all $Y\in\D_2$. Then here is a natural isomorphism $(\varphi^*X)^+\cong\varphi^!(X^+),$ for all $X\in\D_1$.
\end{enumerate}

\end{prop}

\begin{proof}
 (1).  We apply Proposition \ref{prop:EW} and Lemma \ref{l:x+} to get the chain of isomorphisms \[(\varphi_*\varphi^*X)^+\cong(\varphi_*\BS_2\otimes X)^+\cong\hom(\varphi_*\BS_2,X^+)\cong\varphi_*\varphi^!(X^+).\]

 (2). For all $X\in\D_1$ and all $Y\in\D_2$ we obtain the isomorphisms
 \[\D_2(Y,(\varphi^*X)^+)\cong\D_2(\varphi^*X,Y^+)\cong\D_1(X,\varphi_*(Y^+)),\] and
 \[\D_2(Y,\varphi^!(X^+))\cong\D_1(\varphi_*Y,X^+)\cong\D_1(X,(\varphi_*Y)^+).\]
 Now the conclusion follows from the hypothesis and the Yoneda lemma.
\end{proof}

\begin{thm}\label{t:cosilt-desc} Assume the hypotheses made in Setting \ref{s:ttcats}. In addition, suppose $\D_1$ is a rigidly-compactly generated closed tensor-triangulated category and the adjunction $\varphi^*\dashv\varphi_*$ is quasi faithfully flat. If $C_1$ is a pure injective object in $\D_1$, that cogenerates $\D_1$ and $C_2=\varphi^!C_1$ is a cosilting (necessarily pure injective) object in $\D_2$, then $C_1$ is cosilting in $\D_1$.
\end{thm}

\begin{proof}
    Since $C_1$ is pure injective, Lemma \ref{l:pi+} implies that $C_1$ is a direct summand in an object of the form $D^+$, for some $D\in\D_1$. By Lemma \ref{l:pure-tri}, the triangle \[D\to\varphi_*\varphi^*D\to E\to D[1]\] is pure in $\D_1$. As we have already noted, the triple $(\D_1,\D_1,(-)^+)$ is a duality triple, therefore the triangle
    \[D^+[-1]\to E^+\to(\varphi_*\varphi^*D)^+\to D^+\] splits, according to \cite[Theorem 3.6]{BW24}. This shows that $D^+$ is a direct summand in $(\varphi_*\varphi^*D)^+$.  We learned in Proposition \ref{p:iso+} that $(\varphi_*\varphi^*D)^+\cong\varphi_*\varphi^!(D^+)$, therefore, we have shown that $D^+$ is a direct summand in $\varphi_*\varphi^!(D^+)$, naturally. More precisely, the counit $\varphi_*\varphi^!(D^+)\to D^+$ of the adjunction $\varphi_*\dashv\varphi^!$ is a split epimorphism. In conclusion, the same property holds for its direct summand, hence $C_1$ is a direct summand in $\varphi_*\varphi^!C_1$.
    Next, we want to verify all conditions C1, C2 and C3 from Proposition \ref{char-cosilt}. We know by Lemma \ref{cond-C}, that $\varphi_*\varphi^!C_1=\varphi_*C_2\in{^\perp C_1[>0]}$, thus $C_1\in{^\perp C_1[>0]}$, proving C1. Condition C2 follows from the pure injectivity of $C_1$ and condition $C3$ is already assumed in the hypothesis.
\end{proof}

For comparing Theorem \ref{t:cosilt-desc} with the corresponding results in \cite{BHM25}, we again use the notations of Example \ref{e:derived}. In \cite[Theorem 5.12]{BHM25} it is shown that if $C_1=T^+$, for a bounded complex of projectives $T_1\in\Der A$, such that $C_2$ is cosilting, then $C_1$ is also cosilting.  If $C_1=T_1^+$, then $C_1$ is automatically pure injective, by Lemma \ref{l:pi+}. Therefore, even when we restrict ourselves to the case of the derived categories, Theorem \ref{t:cosilt-desc} is stronger than the respective descending result in \cite{BHM25}.

Let $(\D,\otimes,\BS)$ be a rigidly-compactly generated triangulated category, and suppose $\BS$ is silting. Denote $(\D^{\leq0}=\BS[<0]^\perp,\D^{>0}=\BS[\geq0]^\perp)$ the t-structure determined by $\BS$. A usual, the shifted t-structure will be denoted by $(\D^{\leq n},\D^{>n})=(\D^{\leq0}[-n],\D^{>0}[-n])$, for all $n\in\Z$; sometimes is convenient to write $\D^{\geq n+1}=\D^{>n}$. We say that $T\in\D$ is {\em $\BS$-bounded}, if there are integers $a\leq b$, such that
$T\in\Add(\BS)[a]*\ldots*\Add(\BS)[b].$ By shifting conveniently  we obtain
\[T[-a]\in\Add(\BS)*\Add(\BS)[1]*\ldots*\Add(\BS)[b-a]\in\D^{>-l}\cap\D^{\leq0},\] where $l=b-a+1$ is the length of $T$.
Replacing $T$ with its shift as before, we always can consider that an $\BS$-bounded complex appears on the last form.
Notice that $T\in\Add(\BS)*\ldots*\Add(\BS)[l-1]$, for some $l>0$ means that there is a chain of triangles
\[P^{-i}[i]\to T_i\to T_{i+1}\to P^{-i}[i+1],\ (0\leq i<l),\]
with $P^{-i}\in\Add(\BS)$, for all $0\leq i<l$, $T_0=T$, and $T_{l}=0$.  Applying the duality functor $(-)^+$ to this chain of triangles, we obtain
$T^+\in\Prod(\BS^+)[1-l]*\ldots*\Prod(\BS^+)$. We say that $T^+$ is $\BS^+$-cobounded.

\begin{prop}\label{l:s&s+}
Let $(\D,\otimes,\BS)$ be a rigidly-compactly generated triangulated category. Then $\BS^\perp={^\perp\BS^+}$. In particular, $\BS$ is silting if and only if $\BS^+$ is cosilting, and in this case the respective silting and cosilting t-structures coincide, up to a shift with $\pm1$; more precisely $({^\perp\BS^+}[\leq0],{^\perp\BS^+}[>0])=(\D^{\leq-1},\D^{>-1})$.
\end{prop}

\begin{proof} For $X\in\D$ we have \[\D(\BS,X)=0\Leftrightarrow\D(\BS,\BS\otimes X)=0\Leftrightarrow\Hom_\Z(\D(\BS,\BS\otimes X),\Q/\Z)=0\Leftrightarrow\D(X,\BS^+)=0.\]
Further, for an object $X\in\D$ we reformulate the relation $X\in\D^{\leq-1}$ in turn as:  $X[-1+n]\in\BS^\perp$, for all $n>0$, then $X[-1+n]\in{^\perp\BS^+}$, for all $n>0$, and finally $X\in{^\perp\BS^+}[\leq0]$, and similar for the coaisle.
\end{proof}

Recall that a class $\U$ in a compactly generated triangulated category  $\D$ is called {\em definable} if there is a set $\{f_i:C_i\to D_i\mid i\in I\}$ of maps in category $\D^c$ (of compact objects) such that \[\U=\{X\in\D\mid \D(X,f_i)\hbox{ is surjective for all }i\in I\}.\]  In the next Lemma we will need the fact that, if $\U$ is definable then for $X\in\D$ we have $X\in\U$ \iff $X^{++}\in\U$, see \cite[Theorem 6.16]{BW24}.  

\begin{lemma}\label{l:T+}
Let $(\D,\otimes,\BS)$ be a rigidly-compactly generated triangulated category, such that $\BS$ is a silting object.
If $T\in\D$ is a silting, $\BS$-bounded object of finite type, and $C=T^+$, then $C$ is a pure injective $\BS^+$-cobounded cosilting object.
\end{lemma}

\begin{proof}
    In view of Lemma \ref{l:pi+} the pure injectivity of $C$ is immediate. Moreover because $T$ is $\BS$-bounded, the dual $C=T^+$ is $\BS^+$-cobounded as we noticed above.
   Observe that for $X\in\D$ and $N\subseteq\Z$ the chain of isomorphisms
   \[\D\left(X,C[N]\right)=\D\left(X,T^+[N]\right)\cong\D\left(X,T[-N]^+\right)\cong\D\left(T[-N],X^{+}\right)\]
   shows that $X\in{^\perp{C[N]}}$ if and only if $X^+\in T[-N]^\perp$.
   For showing that $C$ is cosilting, we want to apply \cite[Theorem 2.8]{Br24}. In order to do that, we have to verify three conditions:

    Because $T$ is of finite type, it follows that $T[<0]^\perp$ is definable.
    This implies that it is closed under taking double dual, in particular, since $T^{(I)}$ belongs to $T[<0]^\perp$, for any set $I$, we have $\left(T^{(I)}\right)^{++}\in T[<0]^\perp$, and the observation above tells us that $C^I=\left(T^{(I)}\right)^+\in{^\perp{C[>0]}}$. It follows $\Prod(C)\subseteq{^\perp{C[>0]}}$, that is the first condition.

    Further, if $X\in{^\perp{C[\Z]}}$, then $X^+\in T[\Z]^\perp=\{0\}$, so $X=0$, proving the second condition, namely ${^\perp{C[\Z]}}=\{0\}$.

    For the last condition, since $C$ is $\BS^+$-cobounded, we may assume as above that $C\in\Prod(\BS^+)[1-l]*\ldots*\Prod(\BS^+)$. Using Lemma \ref{l:s&s+}, we deduce $C\in\D^{<l}\cap\D^{\geq 0}$.
    Thus $C$ belongs to the coaisle $\D^{\geq0}$, and $C\in{^\perp\D^{\geq l}}={^\perp\D^{\geq0}[-l]}$, therefore we are done.

\end{proof}

Let $\D$ be a rigidly-compactly generated tt category. We say that the duality functor induces a bijection between between equivalent classes of bounded silting and cosilting objects, of the assignment $X\mapsto X^+$ induces a bjection between the class of $\BS$-bounded objects, up to equivalence, and the class of $\BS^+$-bounded cosilting objects, up to equivalence. Remark that by \cite[Theorem 3.3]{AH21} this is the case for $\D=\Der A$, for some commutative ring $A$.

\begin{thm} Assume both hypotheses of Setting \ref{s:tcats} (LR) and Setting \ref{s:ttcats}. In addition, suppose the categories $\D_1$ and $\D_2$ to be rigidly-compactly generated, $\BS_1$ and $\BS_2$ are silting objects, the adjunction $\varphi^*\dashv\varphi_*$ is quasi faithfully flat, and $\varphi_*(Y^+)\cong(\varphi_*Y)^+$, for all $Y\in\D_2$. Finally, suppose that the duality functor induces a bijection between equivalent classes of bounded silting and cosilting objects in both categories $D_1$ and $\D_2$. If $T_1$ is an $\BS_1$-bounded
object such that $T_2$ is a silting object of finite type, then there is an $\BS_1$-bounded silting object in $U_1\in\D_1$, such that $U_2=\varphi^*U_1$ is a silting object in $\D_2$ that is equibvalent to $T_2$.
\end{thm}

\begin{proof}
   First note that if we apply $\varphi^*$ to the finite chain of triangles expressing the fact that $T_1$ is $\BS_1$-bounded, we deduce that $T_2$ is $\BS_2$-bounded. Let $C_1=T_1^+\in\D_1$ and $C_2=T_2^+\in\D_2$. From Lemma \ref{l:T+} we learned that $C_2$ is a pure injective $\BS_2^+$-cobounded cosilting object. Further, Proposition \ref{p:iso+} tells us that
   $\varphi^!C_1=\varphi^!(T_1^+)\cong(\varphi^*T_1)^+=T_2^+=C_2$, and Theorem \ref{t:cosilt-desc} implies that $C_1$ is cosilting. The bijection between bounded silting and cosilting objects in $\D_1$ gives us an $\BS_1$-bounded silting object $U_1\in\D_1$ such that $U_1^+=C_1$. We use Theorem \ref{thm:silting-ascends} to conclude that $U_2=\varphi^*U_1$ is silting in $D_2$; the same argument as for $T_2$, shows that $U_2$ is $\BS_2$-bounded. Finally $C_2=\varphi^!C_1=\varphi^!(U_1^+)\cong(\varphi^*U_1)^+=U_2^+$, and the bijection between bounded silting and cosilting objects in $\D_2$ tells us that $U_2$ and $T_2$ are equivalent.
\end{proof}

\end{document}